\documentclass{amsart}

\usepackage{ifthen}
\usepackage{graphicx}
\usepackage{amssymb}
\usepackage{enumerate}
\usepackage{url}
\newtheorem{anyprop}{Anyprop}[section]

\newtheorem{theorem}[anyprop]{Theorem}
\newtheorem{lemma}[anyprop]{Lemma}
\newtheorem{proposition}[anyprop]{Proposition}
\newtheorem{corollary}[anyprop]{Corollary}

\theoremstyle{definition}

\newtheorem{definition}[anyprop]{Definition}

\newtheorem{remark}[anyprop]{Remark}

\renewcommand  {\ker }  {\operatorname{ker}}

\theoremstyle{remark}

\numberwithin{equation}{section}

\usepackage{amscd}
\usepackage{amssymb}
\input xy
\xyoption{all}

\begin{document}
\title[ON GENERALIZATIONS OF THE DST AND KOH'S CONJECTURE]
{GENERALIZATIONS OF THE DIRECT SUMMAND THEOREM OVER UFD-S FOR SOME BIGENERATED EXTENSIONS AND AN ASYMPTOTIC VERSION OF KOH'S CONJECTURE}


\author[Danny de Jes\'us G\'omez-Ram\'irez]{Danny A. de Jes\'us G\'omez-Ram\'irez}
\author[Edisson Gallego]{Edisson Gallego}
\author[Juan D. Velez]{Juan D. V\'elez}
\address{Vienna University of Technology, Institute of Discrete Mathematics and Geometry,
wiedner Hauptstaße 8-10, 1040, Vienna, Austria.}
\address{University of Antioquia, Calle 67 \# 53-108, Medell\'in, Colombia}
\address{Universidad Nacional de Colombia, Escuela de Matem\'aticas,
Calle 59A No 63 - 20, N\'ucleo El Volador, Medell\'in, Colombia.}
\email{daj.gomezramirez@gmail.com}
\email{egalleg@gmail.com}
\email{jdvelez@unal.edu.co}



\begin{abstract}
This article deals with two different problems in commutative algebra. In
the first part we give a proof of generalized forms of the 
Direct Summand Theorem (DST (or DCS)) for
module-finite extension rings of mixed characteristic $R\subset S$
satisfying the following hypotheses: The base ring $R$ is a Unique
Factorization Domain of mixed characteristic zero. We assume that $S$ is
generated by two elements which satisfy, either radical quadratic equations,
or general quadratic equations under certain arithmetical restrictions.

In the second part of this article, we discuss an asymptotic version of Koh's
Conjecture. We give a model theoretical proof using "non-standard methods".
\end{abstract}

\maketitle


\smallskip

\footnote{This paper should be essentially cited as follows E. Gallego, D. A. J. G\'omez-Ram\'irez and and J. D. V\'elez. The Direct Summand Conjecture for some bi-generated extensions and an asymptotic
Version of Koh's Conjecture. Beitraege zur Algebra und Geometrie (Contributions to Algebra and Geometry) 57(3), 697-712, 2016.}

\section*{Introduction}

The Homological Conjectures have been a focus of research activity since
Jean Pierre Serre introduced the theory of multiplicities in the early 1960s
(\cite{serrelocal}), and since the introduction of characteristic prime
methods in commutative algebra by Peskine, Szpiro, and M. Hochster, in the
mid 1970s \cite{hochsterhomological}, \cite{Peskine}. These conjectures
relate the homological properties of a commutative rings to certain
invariants of the ring structure, as, for instance, its Krull dimension and
its depth. They have been settled for equicharacteristc rings (i.e., rings
for which the characteristic of the ring coincides with that of its residue
field) but many remain open in mixed characteristic \cite{homological
conjectures old and new}. Their validity in mixed characteristic,
nonetheless, is known for rings of Krull dimension less than four.

Among these conjectures the \textit{Direct Summand Conjecture} occupies a
central place, implying, or actually being equivalent, to many of the other%
\textit{\ }conjectures \cite{hochstercanonical}, \cite{ohidirsumcon}, \cite%
{robertsdirsumcon}, \cite{mcculloughthesis}.\newline

\textbf{Direct Summand Conjecture (DSC):} Let $R\subset S$ be a
module-finite extension of noetherian rings (i.e., $S,$ regarded as an $R$%
-module is finitely generated) where $R$ is assumed to be regular. Then, the
inclusion map $R\subset S$ splits as a map of $R$-modules; or equivalently,
there is a retraction $\rho :S\rightarrow R$ from $S$ into $R$. By a
retraction we mean an $R$-linear homomorphism satisfying $\rho (1)=1.$ The
central role of the DSC in Commutative Algebra as well as its relation to
the other homological conjectures is comprehensively explained in \cite%
{hochsterhomological}.

Now, the Direct Summand Conjecture (or D. S. Theorem)
was proved in the general setting by Yves Andr\'{e}, essentially by proving
the (remaining) case of unramified complete regular local rings with the framework
 of perfectoids developed by Scholze \cite{andre}, \cite{bhatt}.

The problem of showing the existence of a retraction $\rho :S\rightarrow R$
may be reduce to the case where $R$ and $S$ are complete local domains \cite%
{hochsterhomological}. Therefore, one may assume that $(R,\mathfrak{m})$ is
in particular a unique factorization domain (UFD) (\cite{eisen}, page 483).
If the Krull dimension of $S$ is $d$ one can always choose a system of
parameters for $S$ contained in $\mathfrak{m}$. By a \emph{system of
parameters} in an arbitrary local ring $(S,\mathfrak{n})$ of Krull dimension 
$d$ we mean a sequence of elements $\{x_{1},...,x_{d}\}$ such that the
radical of the ideal they generate in $S$ is precisely $\mathfrak{n}$, the
unique maximal ideal of $S$ (\cite{eisen}, page 222). It can be proved then
that the DSC holds if and only if for any system of parameters in $S,$ and
any natural number $t>0,$ the \textit{socle element} $(x_{1}\cdots x_{d})^{t}
$ is not contained in the ideal in $S$ generated by the $t+1$ powers of the
parameters. That is, if $(x_{1}\cdots x_{d})^{t}\notin
(x_{1}^{t+1},...,x_{d}^{t+1})S.$ This last statement is known as the \textit{%
Monomial Conjecture}. More precisely:\newline

\textbf{Monomial Conjecture (MC):} Let $(S,\mathfrak{n})$ be any local
noetherian ring of dimension $d$ and let $\{x_{1},...,x_{d}\}$ be a system
of parameters for $S$. Then, for any positive integer $t$, $(x_{1}\cdots
x_{d})^{t}\notin (x_{1}^{t+1},...,x_{d}^{t+1})S.$ This conjecture is
equivalent to the (DSC) \cite{hochstercontracted}. In low dimension, that
is, for rings of Krull dimension $\leq 2$, the DSC, and consequently the MC,
follows as a consequence of the existence of a \emph{normalization} for $S$.
By mapping $S$ into its normalization one may assume that $S$ is a \emph{%
normal} domain, and, for dimensional reasons, a Cohen-Macaulay ring (\cite%
{eisen}, pages 118, 420). Then the Auslander-Buchsbaum formula (\cite{eisen}%
, page 469) implies that $S$ must have projective dimension zero. That is, $S
$ must be an $R$-free module, and consequently the inclusion map
automatically splits.

On the other hand, the general equicharacteristc case, i.e., the case if
which $R$ contains a field of zero characteristic, is handled by elementary
methods: The trace map form the fraction field of $S$ to the fraction field
of $R$ provides a natural retraction. In fact, the inclusion map splits as a
map of $R$-modules under the much weaker hypothesis of $R$ being a normal
domain \cite[Lemma 2]{hochstercontracted}.

If $R$ is equicharacteristc, but contains a field of characteristic $p>0$, a
classical argument given by Hochster (actually, a precursor of his Tight
Closure Theory) shows the validity of the MC by a method in which the
properties of the iterated powers of the Frobenius map are exploited in a
clever way.

It should be remarked that the existence of \textit{Big Cohen Macaulay
Modules} \textit{and Big Cohen Macaulay Algebras} for equicharacteristc
rings immediately implies the validity of the (MC) and the (DSC) \cite%
{homological conjectures old and new}, \cite{CM}.

In the mixed characteristic case it is known that the DSC holds for regular
rings $R$ of Krull dimension $\leq 3.$ The dimension three case was proved
quite recently by R. Heitmann, by means of a rather involved combinatorial
argument \cite{heitmann2002}, \cite{robertsdirsumcon}. This is, undoubtedly, one of the most significant
advances in commutative algebra of the last decades. As mentioned before,
the DSC would follow from the existence of Big Cohen Macaulay Modules or Big
Cohen Macaulay Algebras in mixed characteristic, an open problem in
dimensions greater than three \cite{homological conjectures old and new}.

A natural generalization of the (DSC) was proposed by J. Koh in his doctoral
dissertation \cite{thesis koh}. Koh's question replaces the condition of $R$
being regular for the weaker condition of $S$ having finite projective
dimension as a module over $R$.\newline

\textbf{Koh's Conjecture: }Let $R$ be a noetherian ring, and let $R\subset S$
be a module-finite extension such that $S$, regarded as an $R$-module, has
finite projective dimension. Then there is a retraction from $S$ into $R$.

It is known that many theorems fail when one weakens the hypothesis of $R$
being regular and replaces it by the condition that $R$ is just noetherian,
even a Cohen Macaulay or a Gorenstein complete local domain (see Definition %
\ref{goresnstein}). Then one imposes the condition that the corresponding $R$%
-modules have finite projective dimension (if $R$ is regular this hypothesis
is satisfied automatically, due to Serre's Theorem, \cite{eisen}, Chapter
19). For instance,\textit{\ the Rigidity of Tor} is no longer true in this
context \cite{heitmann}. In a similar manner, the positivity of the
intersection multiplicity $\chi _{R}(M,N)$ for modules $M,N$ of finite
projective dimension over $R$, when $R$ is not regular, is no longer valid 
\cite{hocster-Mclaughing}. Notwithstanding, Koh's conjecture is true for
rings of equal characteristic zero. Unfortunately, it turned out to be false
for equicharacteristc rings of prime characteristic as well as in the case
of mixed characteristic \cite{velezsplitting}.

The fact that this conjecture is true in characteristic zero suggests that
it may be true \textquotedblleft asymptotically\textquotedblright . By an
asymptotic version we mean the following: Given any bound $b>0$ for the
\textquotedblleft complexity\textquotedblright\ of the extension, a notion
we will define in a precise manner in (\ref{complexity}), the set $S_{b}$ of
prime numbers for which there are counterexamples whose characteristic lie
in $S_{b}$ must be \emph{finite} (Theorem \ref{muymuybueno}). We will prove
this asymptotic form for rings that are localization at prime ideals of
affine $k$-algebras, where $k$ is a an algebraically closed field. We
achieve this by first formulating \emph{Koh's Conjecture} as a first order
sentence in the language of rings and algebraically closed fields. Then we
give a proof via \textit{Lefschetz's Principle}\emph{.} A main reference for
the model theoretical methods involved is \cite{Libro schoutens}. Also \cite%
{bounds} may be consulted for a more succinct account of \textquotedblleft
nonstandard methods in commutative algebra.\newline

In the first part of this article we provide a proof of the DSC for
module-finite extensions of rings $R\subset S$ satisfying certain
conditions. We will assume $R$ to be a UFD., where the most interesting case
will be when $R$ is a ring of mixed characteristic zero. On the other hand, $%
S$ will be a module-finite extension generated as a $R$-algebra by two
elements satisfying, either radical quadratic equations (Theorem \ref{dany 1}%
), or satisfying general quadratic equations, under certain arithmetical
restrictions (Theorem \ref{nonradicalsplit}).

In the second part of this article we discuss an asymptotic version of Koh's
Conjecture. We will develop a Model Theoretical approach using
\textquotedblleft non-standard methods\textquotedblright\ similar to those
developed by H. Schoutens \cite{bounds}. The main result of this section is
Theorem \ref{muymuybueno}.

\textit{All rings will be commutative, with identity element }$1$\textit{,
and all modules will be assumed to be unitary.}

\section{The DSC for some radical quadratic extensions}

\subsection{Some reductions\label{reducciones}}

Let $R$ be a UFD, and let us denote by $L$ its fraction field. Let $S$ be a
module-finite extension of $R$ such that $S$ is generated as an $R$-algebra
by two elements $s_{1}$ and $s_{2}$ that satisfy monic polynomials $f_{1}(x)$
and $f_{2}(x)$ in $R[x],$ respectively.

Let us first see that, without loss of generality, we may assume that $%
f_{1}(x)$ and $f_{2}(x)$ have degree greater than one. For if one of them,
for instance $f_{1}(x)$, had degree one then the element $s_{1}$ would be
already in $R.$ In this case $S=R[s_{2}]$. By mapping $C=R[x]/(f_{2}(x))$
onto $S$ we can represent $S$ as a quotient of the form $C/J$, where $J$ is
an ideal of $C$ of height zero. This is because the Krull dimension of $S$
and $T$ is the same and equal to the Krull dimension of $R$, since both
rings are module-finite extensions of $R~$(\cite{kunz}, Corollary 2.13, page
47). \label{kunzito}

Thus, $J$ would be contained in some minimal prime $P$ of $C$. Since $R[x]$
is also a UFD, $P$ is generated by a monic prime factor $p(x)$ of $f_{2}(x),$
hence $J\subset (p(x)).$ But in order to find a $R$-retraction from $S$ into 
$R$ it suffices to find any retraction \textquotedblleft further
above\textquotedblright , $\rho :C/P\rightarrow R$. This is because the
composition of the canonical map $S=C/J\rightarrow C/P$ with $\rho $ then
provides a retraction from $S$ into $R.$ But the map $R\rightarrow C/P\simeq
R[x]/(p(x))$ splits, since $R[x]/(p(x))$ is free as an $R$-module and
consequently $\rho $ can be taken as the projection onto $R$.

Let us then assume that the degrees of $f_{1}$ and $f_{2}$ are greater that
one and henceforth $S$ is minimally generated as an $R$-algebra by $%
s_{1},s_{2}\in S$. Set $T=R[x_{1},x_{2}]/I,$ with $I=(f_{1}(x_{1}),$ $%
f_{2}(x_{2}))$, where $f_{1}(x_{1})$ and $f_{2}(x_{2})$ are monic
polynomials for $s_{1}$ and $s_{2}$, respectively. It is easy to see that $T$
is a free $R$-module, because $T\cong R[x_{1}]/(f_{1})\otimes
_{R}R[x_{2}]/(f_{2})$. In fact, an $R$-basis for $T$ consists of monomials
of the form 
\begin{equation}
B=\{\overline{x}_{1}^{d_{1}}\overline{x}_{2}^{d_{2}}\text{, with }0\leq
d_{i}<\deg f_{i}\}.  \label{basis}
\end{equation}

Let $\varphi :T\rightarrow S$ be the surjective $R$-homomorphism defined by
ending $x_{i}$ into $s_{i}$. Let $J$ denote its kernel, so that $S\cong T/J$%
, where the height of $J$ must be zero, and therefore $J$ must be contained
in some minimal prime of $T$.

The next lemma analyzes the case when $T$ turns out to be a domain, i.e.,
when $J=(0)$. In this case the existence of a retraction $\rho :S\rightarrow
R$ follows automatically, since  $S=T$ is free as an $R$-module.

In what follows $i$ and $j$ denote a pair of indices $1\leq i,$ $j\leq 2,$
with $i\neq j$.

\begin{lemma}
\label{domaincriteria}Let $R$, $T$, $f_{1}$, $f_{2}$ be as above. Let $%
E_{i}=L[x_{i}]/(f_{i}),$ and $F_{j}=E_{i}[x_{j}]/(f_{j})$. Then $T$ is a
domain if and only if both $E_{i}$ and $F_{j}$ are fields. That is, if and
only if $f_{i}$ is irreducible in $L[x_{i}]$ and $f_{j}$ is irreducible in $%
E_{i}[x_{j}]$.
\end{lemma}

\begin{proof}
First, we observe that $L\otimes _{R}T\cong L[x_{1},x_{2}]/I\cong
E_{i}[x_{j}]/(f_{j})=F_{j}$ and that the natural homomorphism $\mu
:T\hookrightarrow L\otimes _{R}T$ is an injection. This is because $T$ is a
torsion free $R$-module, since $R$ is a domain and $T$ is a free module.
Therefore, $T$ is a subring of $F_{j}$, and if $F_{j}$ is a field, $T$ must
be a domain. This gives the \textquotedblleft only if\textquotedblright\
part of the lemma. 

Conversely, let us assume that $T$ is a domain. Arguing by contradiction let
us suppose that either $E_{i}$ or $F_{j}$ is not a field. In the first case
there are monic polynomials of positive degree, $g_{1}$ and $g_{2}$ in $%
L[T_{i}],$ such that $f_{i}=g_{1}g_{2}$ with $\mathrm{deg}(g_{s})<\mathrm{deg%
}(f_{i})$. Now, let $\alpha \in R\smallsetminus \{0\}$ be a common
denominator for the coefficients of $g_{1}$ and $g_{2}$. The equality $%
\alpha ^{2}f_{i}=(\alpha g_{1})(\alpha g_{2})$ in $R[x_{i}]$ implies that $%
\alpha g_{i}$ are zerodivisors in $T$. Besides, $\alpha g_{i}=\alpha 
\overline{x}_{i}^{\mathrm{deg}(g_{i})}+\cdots $ cannot be zero because $%
g_{i},$ written in the $R$-basis \ref{basis} of $T$ has at least one
coefficient different from zero. Therefore, $T$ would not be a domain, a
contradiction. Then we may assume that $E_{i}$ is a field. \newline
On the other hand, if $f_{j}$ were reducible over $E_{i}[x_{j}]$ we could
write $f_{j}=h_{1}h_{2}$, where $h_{1},h_{2}\in E_{i}[x_{j}]$ are monic
polynomials of degree less than $\mathrm{deg}(f_{j})$. Let us choose $%
\widetilde{h}_{1},\widetilde{h}_{2}$ in $L[x_{1},x_{2}]$ such that $\psi (%
\widetilde{h}_{s})=h_{s}$, $s=1,2,$ where $\psi :L[x_{i},x_{j}]\rightarrow
E_{i}[x_{j}]$ is the natural homomorphism induced by the projection map $%
L[x_{i}]\rightarrow E_{i}$. In fact, we can choose each $\widetilde{h}_{s}$,
considered as a polynomial in $(L[x_{i}])[x_{j}]$, such that each of its
coefficients in $L[x_{i}]$ is a polynomial in $x_{i}$ with degree less than $%
\mathrm{deg}(f_{i})$. Hence, there exists $\widetilde{h}_{3}$ in $%
L[x_{1},x_{2}]$ such that $f_{j}-\widetilde{h}_{1}\widetilde{h}_{2}=%
\widetilde{h}_{3}f_{i}$. Choose any nonzero element $c\in R$ such that $c%
\widetilde{h}_{r}\in R[x_{1},x_{2}]$, for $r=1,2,3$. Then we have that $c%
\widetilde{h}_{1}c\widetilde{h}_{2}=c^{2}f_{j}-c(c\widetilde{h}_{3})f_{i}\in
I$ and consequently the classes of $c\widetilde{h}_{1}$ and $c\widetilde{h}%
_{2}$ in $T$ must be different from zero. Thus, $T$ would not be a domain,
which is a contradiction. This proves $f_{j}$ must be irreducible over $%
E_{i}[x_{j}].$
\end{proof}

\begin{corollary}
\label{CDU}Let $R$ be a UFD where its field of fractions $L$ has
characteristic different from two. Assume that $f_{i}=x_{i}^{2}-a_{i}$ are
irreducible polynomials in $L[x_{i}]$, $i=1,2$. If $%
T=R[x_{1},x_{2}]/(f_{1},f_{2})$ is not a domain then there exist nonzero
elements $c,d,u$ in $R$ such that $a_{1}=d^{2}u,a_{2}=c^{2}u,$ where $c,d$
are relatively prime.
\end{corollary}

\begin{proof}
Since $T$ is not a domain, by Lemma \ref{domaincriteria} we may assume
without loss of generality that one of the polynomials $f_{i},$ for instance 
$f_{2}(x_{2}),$ is reducible in $E_{1}[x_{2}]$. But this is equivalent to
saying that $f_{2}(x_{2})$ has a root $e\in E_{1}$, that we may write as $%
e=e_{1}+e_{2}\overline{x_{1}}$, where $e_{1},e_{2}\in L,$ and $\overline{%
x_{1}}^{2}=a_{1}$. Hence 
\begin{equation*}
a_{2}=e^{2}=(e_{1}^{2}+e_{2}^{2}a_{1})+2e_{1}e_{2}\overline{x_{1}}.
\end{equation*}%
Then $a_{2}=e_{1}^{2}+e_{2}^{2}a_{1}$ and $2e_{1}e_{2}=0$. But $\mathrm{char(%
}L)\neq 2$ implies $e_{1}e_{2}=0$. If $e_{2}=0$ then $a_{2}=e_{1}^{2}$ and
therefore $f_{2}=(x_{2}+e_{1})(x_{2}-e_{1})$, which is a contradiction.
Thus, $e_{1}=0$ and $a_{2}=e_{2}^{2}a_{1}$.

Now write $e_{2}=e/d$, where $c,d\neq 0$ are relatively prime elements in $R$%
. So $d^{2}a_{2}=e^{2}a_{1};$ but $d^{2}$ does not divide $c^{2}$ and 
consequently $d^{2}$ divides $a_{1}.$ Hence, there is $u\in R$ such that $%
a_{1}=d^{2}u$. Replacing $a_{1}$ in $d^{2}a_{2}=e^{2}a_{1}$ gives the
equation $d^{2}a_{2}=c^{2}d^{2}u$. After dividing by $d^{2}\neq 0$ we obtain 
$a_{2}=c^{2}u$, which proves the corollary.
\end{proof}

\begin{lemma}
\label{minimalprimesI}Let $R$ be a UFD, let $B=R[x,y]$, and let $u,c,d$ be
elements in $R$ different from zero. Define $f_{1}=x^{2}-d^{2}u$, $%
f_{2}=y^{2}-c^{2}u$, and set $I=(f_{1},f_{2}),$ the ideal in $B$ generated
by $f_{1}$ and $f_{2}$. Assume that $\{c,d\}$ is a regular sequence in $R$ (%
\cite{eisen}, page 173). Then the minimal prime ideals of $I$ are $%
P_{r}=(f_{1},f_{2},f_{3,r},f_{4,r})$, where $f_{3,r}=dy+(-1)^{r}cx,$ $%
f_{4,r}=xy+(-1)^{r}cdu$, $r=0,1$.
\end{lemma}

\begin{proof}
Let us introduce a new variable $z$, and let $g(z)=z^{2}-u$ in $R[z].$
Clearly $g$ has no roots in $R$, since $f_{1}$ is irreducible. Therefore, $g$
is also irreducible and consequently the ideal $(g)$ is prime in $R[z]$;
this is because $R$ is a UFD. Define $\psi _{r}:B\rightarrow R[z]/(g)$ as
the unique $R$-homomorphism that sends $x$ into $d\overline{z}$ and $y$ into 
$(-1)^{r+1}c\overline{z}$, where $\overline{z}$ denotes the class of $z$ in $%
R[z]/(g).$ We prove that $\mathrm{ker}(\psi _{r})=P_{r}$. 

First, we observe that $P_{r}\subset \mathrm{ker}(\psi _{r})$, since: 
\begin{equation*}
\psi _{r}(f_{1})=d^{2}\overline{z}^{2}-d^{2}u=d^{2}u-d^{2}u=0,
\end{equation*}%
\begin{equation*}
\psi _{r}(f_{2})=c^{2}\overline{z}^{2}-c^{2}u=0,
\end{equation*}%
\begin{equation*}
\psi _{r}(f_{3,r})=\psi _{r}(dy+(-1)^{r}cx)=d(-1)^{r+1}c\overline{z}%
+(-1)^{r}cd\overline{z}=0,
\end{equation*}%
\begin{equation*}
\psi _{r}(f_{4,r})=\psi _{r}(xy+(-1)^{r}cdu)=(-1)^{r+1}cd\overline{z}%
^{2}+(-1)^{r}cdu=0.
\end{equation*}%
To prove that $\mathrm{ker}(\psi _{r})\subset P_{r}$ we use the well known
fact that in a polynomial ring in one variable with coefficients in a
commutative ring the ordinary division algorithm holds, as long as the
divisor is a monic polynomial. This justifies the following procedure:

Let $h(x,y)$ be a polynomial in $\ker (\psi _{r})$. If we regard $h(x,y)$ as
a polynomial in the variable $y$ with coefficients in $R[x]$ then, dividing
by $f_{1}=x^{2}-d^{2}u$ we can write $h(x,y)$ as 
\begin{equation}
h(x,y)=f_{1}q(x,y)+q_{1}(y)x+q_{0}(y),  \label{ee1}
\end{equation}%
where $q(x,y)\in B$ and $q_{0}(y),q_{1}(y)\in R[y]$.

Similarly, after dividing $q_{0}(y)$ by $f_{2}=y^{2}-c^{2}u$ we may find $%
q_{2}(y)\in R[y],$ and $a_{1},a_{2}$ elements in $R$ such that 
\begin{equation}
q_{0}(y)=f_{2}q_{2}(y)+(a_{1}y+a_{2}).  \label{ee2}
\end{equation}%
Similarly, there exists polynomials $q_{3}(y),$ $q_{4}(y)$ in $R[y],$ and $%
b_{1},b_{2},\in R$ such that 
\begin{equation}
q_{1}(y)x=f_{4,r}q_{3}(y)+q_{4}(y)+b_{1}x+b_{2}.  \label{ee3}
\end{equation}%
Dividing $q_{4}(y)$ by $f_{2}$ we obtain: 
\begin{equation}
q_{4}(y)=q_{5}(y)f_{2}+e_{1}y+e_{2},  \label{ee4}
\end{equation}%
for certain polynomial $q_{5}(y)$, and certain elements $e_{1},e_{2}$ in $R$%
. 

Replacing equations (\ref{ee2}), (\ref{ee3}) and (\ref{ee4}) in (\ref{ee1})
we can write $h(x,y)$ as 
\begin{equation*}
h(x,y)=f_{1}q(x,y)+f_{4,r}q_{3}(y)+(q_{5}(y)+q_{2}(y))f_{2}+l(x,y)\text{, }
\end{equation*}%
where $l(x,y)$ is the linear polynomial 
\begin{equation*}
l(x,y)=v_{1}y+b_{1}x+v_{2},
\end{equation*}%
where $v_{1}=e_{1}+a_{1}$ and $v_{2}=e_{2}+a_{2}+b_{2}$. From the condition $%
\psi _{r}(h)=0$ we get: 
\begin{equation*}
\psi _{r}(h)=\psi _{r}(l)=v_{1}(-1)^{r+1}c\overline{z}+b_{1}d\overline{z}%
+v_{2}=0.
\end{equation*}%
Since $R[z]/(g)$ is a $R$-free module with basis $\{1,\overline{z}\}$ we
then must have: 
\begin{equation}
(-1)^{r+1}v_{1}c+b_{1}d=0\text{ and }v_{2}=0.  \label{ee5}
\end{equation}%
But$\{c,d\}$ is a regular sequence in $R$, hence there must be some $u\in R$
such that $b_{1}=uc$. Since $R$ is a domain we obtain from \ref{ee5} that $%
(-1)^{r}d=v_{1}$, and consequently 
\begin{equation*}
l(x,y)=v_{1}y+b_{1}x=u((-1)^{r}dy+cx)=u(-1)^{r}f_{3,r}.
\end{equation*}%
This shows $h\in P_{r}$. 

Clearly $P_{0}$ and $P_{1}$ are prime ideal of $B$ because $B/P_{r}$ is an
integral domain isomorphic to $R[z]/(g)$.

Finally, let us show that $P_{0}$ and $P_{1}$ are the the only minimal
primes of $I.$ Let $Q$ be any other minimal prime over $I$. Then 
\begin{equation*}
c^{2}f_{1}-d^{2}f_{2}=-(dy-cx)(cx+dy)\in Q,
\end{equation*}%
so $f_{3,1}=dy-cx\in Q,$ or $f_{3,0}=cx+dy\in Q$. 
\end{proof}

\begin{proof}
Let us suppose $f_{3,1}\in Q.$ Hence 
\begin{equation*}
-x(cx-dy)=dxy-cx^{2}=dxy-cd^{2}u=d(xy-cdu)\in Q.
\end{equation*}%
Similarly, 
\begin{equation*}
y(cx-dy)=cxy-dy^{2}=cxy-dc^{2}u=c(xy-cdu)\in Q.
\end{equation*}%
We claim that $f_{4,1}=xy-cdu$ is an element of $Q$. If we suppose
otherwise, the elements $c,d$ would be contained in $Q$. Then we would have
that $x^{2}=(f_{1}+d^{2}u)$ and $y^{2}=f_{2}+c^{2}u$ would also be contained
in $Q.$ Consequently, $x,y$ would also be elements of $Q$. Henceforth, the
ideal $(c,d,x,y)B$ would be contained in $Q$. But each generator of $P_{1}$
is in $(c,d,x,y)B$ and consequently $P_{1}\subset Q$. But this must be a
proper inclusion, since $x\notin P_{1}$ (this is because every monomial
containing $x$ in each one of the generator of $P_{1}$ is either quadratic, $%
x^{2}$ or $xy$, or linear of the form $cx$, but $x\notin (x^{2},xy,cx)B,$
since $1\notin (x,y,c)B$). This contradicts the minimality of $Q,$ and
proves the claim. 

Thus $Q$ must contain $f_{1},f_{2},f_{3,1},f_{4,1}$ and consequently $%
P_{1}=Q.$

The second case, when $f_{3,0}=cx+dy\in Q,$ can be treated in a similar
fashion. When this occurs we obtain that $P_{0}=Q$.
\end{proof}

\begin{theorem}
\label{dany 1}Let $R\subset S$ be a module-finite extension of noetherian
rings such that $R$ is a UFD and such that the characteristic of the
fraction field $L$ of $R$ is different from two. Suppose $S$ is generated as
an $R$-algebra by two elements $s_{1},s_{2}\in S$ satisfying monic radical
quadratic polynomials $f_{1}=x_{1}^{2}-a_{1}$ and $f_{2}=x_{2}^{2}-a_{2}$,
respectively. Then $R\subset S$ splits.
\end{theorem}

\begin{proof}
By the Going Up (\cite{eisen}, page 129) there exists a prime ideal $%
Q\subset S$ that contracts to zero in $R$. Therefore, we may replace $S$ by $%
S/Q$ without altering the relevant hypothesis. As observed in Section \ref%
{reducciones}, it suffices to find a retraction from $S/Q$ into $R$. Hence,
we may reduce to the case where $S$ is a domain.

As observed in that same section we can write $S$ as a quotient $T/J$, where 
$T=R[x_{1},x_{2}]/I$, $I=(f_{1},f_{2}),$ and $J\subset T$ is an ideal of
height zero. Moreover, we can assume each $f_{i}$ to be irreducible in $\in
R[x_{i}]$, otherwise, as noticed in that section, the splitting of $R\subset
S$ would immediately follow from the fact that $S$ would be a free $R$%
-module. 

On the other hand, if $T$ is a domain then $J=0,$ and so $S=T$ would be a
free $R$-free (of rank four) and $R\subset S$ would also split. If on the
contrary, $T$ is not a domain, then by Corollary \ref{CDU} there exist $c,d,u
$ nonzero elements of $R$ such that $a_{1}=d^{2}u,$ $a_{2}=c^{2}u,$ where $%
c,d$ are relatively prime. This implies that $\{c,d\}$ is a regular sequence
in $R$. Then Lemma \ref{minimalprimesI} implies that the minimal primes of $T
$ are precisely $P_{0}/I$ and $P_{1}/I.$Then $J\subset P_{r}/I,$ for some $%
r=0,1$. Hence, if $\alpha :T/J\rightarrow T/(P_{r}/I)$ is the natural map
and $\psi _{r}^{\prime }:T/(P_{r}/I)\rightarrow R[z]/(z^{2}-u)$ is the $R$%
-isomorphism induced by $\psi _{r}$ (as in the proof of Lemma \ref%
{minimalprimesI}) and $\pi _{1}:R[z]/(z^{2}-u)\rightarrow R$ is the $R$%
-module projection on the first component, then a retraction for $R\subset S$
is given by $\rho =\pi _{1}\circ \psi _{r}^{\prime }\circ \alpha $. This
proves the corollary.
\end{proof}

\begin{theorem}
\label{dany 2}Let $R\subset S$ be a module-finite extension of noetherian
rings such that $R$ is a UFD, and such that the characteristic of the
fraction field $L$ of $R$ is different from two. Let us assume that the
element $2$ is a unit in $R,$ and that $S$ is generated as an $R$-algebra by
two elements $s_{1},s_{2}\in S$ satisfying monic quadratic polynomials $%
f_{1}=x^{2}-ax+b$ and $f_{2}=y^{2}-cy+d$, respectively. Then $R\subset S$
splits.
\end{theorem}

\begin{proof}
We can reduce to the radical quadratic case, as in Theorem \ref{dany 1}, by
\textquotedblleft completing the squares\textquotedblright . That is, define
rings 
\begin{equation*}
T^{\prime }=R[u,v]/(u^{2}+b-a^{2}/4,\text{ }v^{2}+c-d^{2}/4)
\end{equation*}%
and%
\begin{equation*}
T=R[x,y]/(x^{2}-ax+b,\text{ }y^{2}-cy+d).
\end{equation*}%
Let $\psi :T^{\prime }\rightarrow T$ be the linear isomorphism that sends
the variable $u$ to $x-a/2$ and the variable $v$ to $y-c/2$.

We already know that $S$ is isomorphic to a quotient of $T$ hence it is also
isomorphic to a quotient of $T^{\prime }$. Then, we may choose new
generators of $S$ as $R$-algebra satisfying monic radical quadratic
equations: Namely, the classes of $\overline{u}$ and $\overline{v}$. Thus,
by Theorem \ref{dany 1}, the extension $R\subset S$ must splits.
\end{proof}

\section{The DSC for some nonradical quadratic extensions}

Our next goal is to prove the following result.

\begin{theorem}
\label{nonradicalsplit}Let $R$ be a UFD such that such that the
characteristic of the fraction field $L$ of $R$ is different from two. Let $%
R\subset S$ be a module-finite extension such that $S$ is minimally
generated as an $R$-algebra by elements $s_{1},s_{2}\in S$. Let us assume $%
f(s_{1})=g(s_{2})=0$, where $f(x)=x^{2}-ax+b$ and $g(y)=y^{2}-cy+d,$ for
some $a,b,c,d\in R$. If $\mathrm{\gcd }(2,c)=1$ and $a^{2}-4b$ is square
free, then $R\subset S$ splits.
\end{theorem}

In order to prove this Theorem we need the following lemma:

\begin{lemma}
\label{anetrior}Let $R$ be a UFD such that the characteristic of the
fraction field $L$ of $R$ is different from two. Let $T=R[x,y]/(f(x),g(y))$,
where $f(x)=x^{2}-ax+b$ and $g(y)=y^{2}-cy+d,$ for some $a,b,c,d\in R$.
Suppose that $\mathrm{\gcd }(2,c)=1$, that the discriminant of $f(x),$ $%
a^{2}-4b\neq 0,$ is square free in $R$ and that $f(x)$ is irreducible. 
If $T$ is not a domain, then there exists $e\in R$ such that $(c\pm ae)/2\in R$.
In this case the minimal primes of $T$ are $P_{1}=(\overline{h}_{1})$ and 
$P_{2}=(\overline{h}_{2})$, where $\overline{h}_{1}$ and $\overline{h}_{2}$ 
are the classes in $T$ of the polynomials $h_{1}(x,y)=y-ex-(c-ae)/2$ and 
$h_{2}(x,y)=y-ex-(c+ae)/2.$
\end{lemma}

\begin{proof}
If we assume that $T$ is not a domain, then, by Lemma \ref{domaincriteria}, $g$
must be reducible in $E[y]$, where $E$ denotes the field $L[x]/(f(x))$. It
is clear that $E$, as a field, is isomorphic to the extension field $%
L(u^{1/2})$, for $u=a^{2}-4b$. Therefore, $g(y)$ has a root $\gamma =\alpha
+\beta u^{1/2}$ in $E$, since $T\cong E[y]/(g(x))$ is not a domain. 
But one can verify directly that the conjugate $%
\overline{\gamma }=\alpha -\beta u^{1/2}$ is also a root of $g$. Thus, $%
g=(y-\gamma )(y-\overline{\gamma })$. By comparing coefficients we get $%
\alpha ^{2}-\beta ^{2}u=d$ and $c=2\alpha $. Hence, $4d=c^{2}-4\beta ^{2}u$. 

Let us write $\beta =q/r$, for $q,r\in R$ such that $\mathrm{\gcd }(q,r)=1$.
From $4d=c^{2}-4\beta ^{2}u$ we obtain that $4r^{2}d=r^{2}c^{2}-4q^{2}u$.
Then, $4(r^{2}d+q^{2}u)=r^{2}c^{2}$. This implies that $4\mid r^{2}c^{2}$;
but $\mathrm{\gcd }(2,c)=1$, therefore $4\mid r^{2},$ and so $2\mid r$.
Write $r=2t$, for some $t\in R\smallsetminus \{0\}$. Thus,%
\begin{equation}
4(r^{2}d+q^{2}u)=4t^{2}c^{2}  \label{eee}
\end{equation}

After dividing by $4$ equation \ref{eee} we obtain: $%
4t^{2}d+q^{2}u=t^{2}c^{2}.$ Or, equivalently, $t^{2}(c^{2}-4d)=q^{2}u$. From
this, it follows that $t^{2}\mid q^{2}u$, which in turn implies $t^{2}\mid u$%
, because $\mathrm{\gcd }(t,q)=1$. But $u$ is square free, therefore $t$
must be a unit. Then we may assume that $q/t\in R$. If we let $e=q/t$, then
the equation $t^{2}(c^{2}-4d)=q^{2}u$ can be rewritten as:  
\begin{equation}
c^{2}-4d=e^{2}(a^{2}-4b).  \label{eq1}
\end{equation}%
We will now prove that $2\mid (c\pm ae)$. In fact, suppose that $2=\prod
p_{i}^{n_{i}}$, and that $c+ae=\prod p_{i}^{m_{i}}$ and $c-ae=\prod
p_{i}^{k_{i}}$ are factorizations into powers of prime elements. By allowing
some exponents to be zero we may assume each product involves the same
primes. We shall see that $n_{i}\leq \mathrm{min}(m_{i},k_{i})$, for all $i$%
. From (\ref{eq1}) it follows that 
\begin{equation*}
((c-ea)/2)((c+ea)/2)=d-e^{2}b\in R.
\end{equation*}%
This implies that $2n_{i}\leq m_{i}+k_{i}$, since $(c-ea)(c+ea)/2^{2}$
belongs to $R$. Arguing by contradiction, suppose there is $j$ such that $%
n_{j}>\mathrm{min}(m_{j},k_{j})$. Without loss of generality, we may assume $%
m_{j}=\mathrm{min}(m_{j},k_{j})$. Hence, $n_{j}\leq k_{j}$, otherwise $%
2n_{j}>m_{j}+k_{j}$, which is a contradiction. Therefore, $p_{j}^{n_{j}}\mid
(c-ae)$, and consequently 
\begin{equation*}
p_{j}^{n_{j}}\mid (c-ae)+2ae=c+ae
\end{equation*}%
which means that $n_{j}\leq m_{j}$. Thus, $n_j\leq \mathrm{min}(m_{j},k_{j})$.
Summarizing, $n_{i}\leq \mathrm{min}(m_{i},k_{i})$ for all $i$.
Hence, $2\mid (c\pm ae)$.

Now, let us see that $P_{1}=(\overline{h}_{1})$ and $P_{2}=(\overline{h}_{2})
$ are the minimal primes of $T$. Using the fact that $%
((c-ea)/2)((c+ea)/2)=d-e^{2}b$. We see by a direct computation that 
\begin{equation}
h_{1}h_{2}=e^{2}f+g.  \label{eeee}
\end{equation}
Hence, any minimal prime in $T=R[x,y]/(f,g)$ must contain either $\overline{h%
}_{1}$ or $\overline{h}_{2}$. On the other hand, $R[x,y]/(f,g,h_{1})\cong
R[x]/(f),$ since we can eliminate the variable $y$ using $%
h_{1}(x,y)=y-ex-(c-ae)/2,$ by sending $y$ into $ex+(c-ae)/2$. In a similar
fashion we see that $R[x,y]/(f,g,h_{2})\simeq R[x]/(f)$. Since $R[x]/(f(x))$
is a domain, $P_{1}=(\overline{h}_{1})$ and $P_{2}=(\overline{h}_{2})$ must
be prime ideals of $T$ and since each minimal prime of $T$ must contain
either $\overline{h}_{1}$ or $\overline{h}_{2},$ then these must be the only
minimal primes.
\end{proof}

Now, we are ready to prove Theorem \ref{nonradicalsplit}.

\begin{proof}
As we observed in Section \ref{reducciones} we may assume $S$ is minimally
generated as an $R$-algebra by certain elements $s_{1}$ and $s_{2}$ whose
corresponding monic polynomials $f(x)$ and $g(y)$ over $R$ have degree
greater than one. As we observed in that same section, this implies that $%
f(x)$ is irreducible. 

We know that $S$ can be represented as a quotient of the form $T/J$, where $%
T=R[x,y]/(f(x),g(y))$ and $J\subset T$ is an ideal of $T$ of height zero. By
Lemma \ref{anetrior}, $J\subset P_{1}=(\overline{h}_{1})$ or $J\subset
P_{2}=(\overline{h}_{2}),$ the minimal prime ideals of $T$ defined in Lemma %
\ref{anetrior}, with $h_{1}(x,y)=y-ex-(c-ae)/2$ and $%
h_{2}(x,y)=y-ex-(c+ae)/2.$ We can obtain the desired retraction $\rho
:S\rightarrow R$ as the composition of the following natural chain of $R$%
-homomorphisms: 
\begin{equation*}
S=T/J\rightarrow T/P_{j}\overset{\varphi }{\rightarrow }R[x]/(f(x))%
\rightarrow R\oplus R\overline{x}\overset{\pi _{1}}{\rightarrow }R,
\end{equation*}%
where $\varphi $ is the $R$-homomorphism defined by defining $x$ into $x$
and $y$ into $h_{j}-y$, and $\pi _{1}$ is canonical projection on $R$. 
\end{proof}

\section{An asymptotic form of Koh's conjecture}

\subsection{\textbf{Lefschetz's Principle}}

In this second part of this article we give an \textit{asymptotic}
formulation of Koh's Conjecture as well as its proof. Let us recall that
this conjecture states that if $R$ is a Noetherian ring and that if $%
R\subset S$ a module-finite extension of rings such that the projective
dimension of $S$ as an $R$-module is finite, then there exists a retraction $%
\rho :S\rightarrow R$.

The asymptotic form of Koh's conjecture is the following: given any bound $%
b>0$ for the \textquotedblleft complexity\textquotedblright\ of the
extension (see Definition \ref{complexity}) the set $S_{b}$ of prime numbers
for which there are counterexamples whose characteristic lie in $S_{b}$ must
be \emph{finite}. We will prove this asymptotic form for rings that are
localization at prime ideals of affine $k$-algebras, where $k$ is a an
algebraically closed field. We refer to such rings by the shorter name of%
\emph{\ local }$k$\emph{-algebras. }

Let us begin by recalling Lefschetz's principle:

\begin{theorem}[Lefschetz's Principle]
\label{lef-fuerte} Let $\phi $ be a sentence in the language of rings. The
following statements are equivalent.

\begin{enumerate}
\item The sentence $\phi $ is true in an algebraic closed field of
characteristic zero.

\item There exists a natural number $m$ such that for every $p>m$, $\phi $
is true for every algebraically closed field of characteristic $p$.
\end{enumerate}
\end{theorem}

\begin{proof}
See \cite{modelteory}, Corollary 2.2.10, page 42.
\end{proof}

We state without proof one of the cornerstones of Model Theory, \emph{the} 
\emph{Compactness Theorem}. This theorem guarantees the existence of a model
for a $\mathcal{L}$-theory $T$ (we say in this case that $T$ is \emph{%
satisfiable}) if and only if there exists a model for each finite subset of $%
T$.

\begin{theorem}[Compactness Theorem]
Suppose $T$ is a $\mathcal{L}$-theory. Then, $T$ is satisfiable if and only
if every finite subset of $T$ is satisfiable.
\end{theorem}

As a consequence of the Compactness Theorem one can readily deduce the
following proposition (\cite{modelteory}, page 42).

\begin{proposition}
\label{tranferencia fuerte} Let $\phi $ be a first order sentence which is
true in every field $k$ of characteristic zero. Then, there exists a prime
number $p_{0}$ such that $\phi $ is true in each field $F$ of characteristic 
$q$, for $q>p_{0}$.
\end{proposition}

\section{Codes for polynomial rings and modules}

Throughout this discussion we will fix a field $k$ and a monomial order in
the polynomial ring $A=k\left[ x_{1},\ldots ,x_{n}\right] $.

\begin{definition}
\label{complexity}Let $R$ be a finitely generated $k$-algebra, and let $I$
be an ideal of $k\left[ x_{1},\ldots ,x_{n}\right] .$ We will say that:

\begin{enumerate}
\item The ideal $I$ has \textbf{complexity} \textbf{at most } $d$, if $n\leq
d$ and it is possible to choose generators\ for $I,$ $f_{1},\ldots ,f_{s},$
with $\deg f_{i}\leq d,$ for $i=1,\ldots ,s$.

\item We say $R$ has \textbf{\ complexity at most } $d$ if there is a
presentation of $R$ as $k\left[ x_{1},\ldots ,x_{n}\right] /I$, with $I$ of
complexity at most $d$.

\item If $J\subset R$ is an ideal, we will say that $J$ \textbf{\ has
complexity at most } $d$, if $R$ has complexity less than or equal to $d,$
and there exists a lifting of $J$ in $k\left[ x_{1},\ldots ,x_{n}\right] $,
let us say $J^{\prime },$ with complexity at most $d$.

\item If $R$ is a local $k$-algebra, we say it has \textbf{\ complexity at
most } $d$ if $R$ can be written as $R=\left( k\left[ x_{1},\ldots ,x_{n}%
\right] /I\right) _{p},$ for some prime ideal $p\subset k\left[ x_{1},\ldots
,x_{n}\right] /I$ such that the complexity of $R$ and $p$ is at most $d$.

\item If $M$ is any finitely generated $R$-module, we will say that $M$ has 
\textbf{\ complexity at most } $d$ if $R$ is a $k$-algebra of complexity at
most $d,$ and there exists an exact sequence $R^{t}\overset{\Gamma }{%
\rightarrow }R^{s}\rightarrow M\rightarrow 0$, with $s,t\leq d$, where all
the entries of the matrix $\Gamma $ are polynomials (or quotients of
polynomials, in the local case) with degree at most $d$.

\item Let $M\subset R^{d}$ be $R$-submodule. We will say that $M$ has 
\textbf{degree type at most }$d$ (written as $gt\left( M\right) \leq d$) if
the complexity of $R$ is at most $d,$ and $M$ is generated by $d$-tuples
with all its entries of degree at most $d.$ If $M$ is a finitely generated $%
R $-module, we will say that $M$ has \textbf{complexity degree at most }$d$
if there exist submodules $N_{2}\subset N_{1}\subset R^{d},$ both of degree
type at most $d,$ such that $M\cong N_{1}/N_{2}$.
\end{enumerate}
\end{definition}

Now, for any polynomial $f\in A$ we will denote by $a_{f}$ the tuple of all
the coefficients of $f$ listed according to the fixed order. When the
complexity of an ideal $I$ is at most $d,$ and $I=(f_{1},\ldots ,f_{s})$,
then $I$ can be encoded by the tuple $a_{I}$ that consists of all the
coefficients of the polynomials $f_{i}$. It is not difficult to see that the
length of this tuple only depends on $d$ \cite{bounds}. On the other hand,
given one of those tuples $a$ we can always reconstruct the ideal where it
comes from, an ideal we shall denote by $\mathcal{I}\left( a\right) $.
Similarly, if $R$ is a $k$-algebra with complexity at most $d$ then $R$ can
be written as $k\left[ x_{1},\ldots ,x_{n}\right] /\mathcal{I}\left(
a\right) .$ We will write this fact as $R=\mathcal{R}\left( a\right) $.

Let $M$ be an $R$-module. If the complexity degree of $M$ is at most $d$
then the minimal number of generator for $M$ is bounded in function of $d$.
Hence $M$ can be encoded by a tuple $v=\left( n_{1},n_{2}\right) ,$ where $%
n_{1}$ is a code for $N_{1}$ and $n_{2}$ is a code for $N_{2}.$ We will
write this as $M\cong \mathcal{M}\left( v\right) $ (see \cite{bounds}).

Finally, if $\phi (\xi )$ is a formula with free variable $\xi $ and
parameters from a ring $R$, then by $a\in |\phi |_{R}$ we will mean $%
R\models \phi (a)$ (\cite{modelteory}, Definition 1.1.6.)

The proof of the following theorem may be found in \cite{bounds}, Remark
2.3, and \cite{Libro schoutens}, Theorem 4.4.1, page 59.

\begin{theorem}
\label{member} Given $d>0$, there exists a formula \textrm{IdMem}$_{d}$
(Ideal Membership) such that for any field $k$, any ideal $I\subset k\left[
x_{1},\ldots ,x_{n}\right] ,$ and any $k$-algebra $R$, both of complexity at
most $d$ over $k$, it holds that $f\in IR$ if an only if $k\models $\textrm{%
IdMem}$_{d}(a_{f},a_{I})$. Here $a_{f}$, and $a_{I}$ denote codes for $f$
and $I,$ respectively.
\end{theorem}

\begin{remark}

\begin{enumerate}
\item \label{Inc} Using Corollary \ref{member}, it is easy to get for each $d
$ formulas \textrm{Inc}$_{d}$ (Inclusion) and \textrm{Equal}$_{d}$
(Equality) such that if $R$ is a finitely generated $k$-algebra with
complexity at most $d,$ and if $J$ and $I$ are ideals of $R$ with complexity
less than $d$, then $(a_{I,}a_{J})\in \left\vert \text{Inc}_{d}\right\vert
_{K}$ (resp. $(a_{I,}a_{J})\in \left\vert \text{Equal}_{d}\right\vert _{K}$)
if and only if $I$ is included in $J$, $I\subset J$, (resp. $I=J.$)

\item \label{Formula-Ideal Maximal} Given $d,n>0$ there exists a formula 
\textrm{MaxIdeal}$_{d,n}$ such that for any algebraic closed field $k$ and
any ideal $\mathfrak{m}\subset k[x_{1},\ldots ,x_{n}]$ of complexity at most 
$d$ we have: $\mathfrak{m}$ is a maximal ideal if and only if $k\models 
\mathrm{MaxIdeal}_{d,n}(a_{m})$, where $a_{m}$ is a code for $\mathfrak{m}$.
In fact, by the Nullstellensatz $\mathfrak{m}$ is maximal if and only if
there exist $b_{1},\ldots ,b_{n}\in k$ such that $\mathfrak{m}%
=(x_{1}-b_{1},\ldots ,x_{n}-b_{n})$. Let us call $J=(x_{1}-b_{1},\ldots
,x_{n}-b_{n})$. Then, the required formula is: 
\begin{equation*}
\mathrm{MaxIdeal}(\xi )=(\exists b_{1},\ldots ,b_{n})(\mathrm{Equal}_{d}(\xi
,a_{J})),
\end{equation*}%
where $\xi $ and $a_{J}$ must be replaced by the codes $a_{m}$ of $\mathfrak{%
m},$ and $a_{J}$ of $J,$ respectively.
\end{enumerate}
\end{remark}

We will also need the following lemma:

\begin{lemma}
(\cite{bounds}, Lemma 3.2) For each $d>0$ there is a bound $D=D(d)$ with the
following property: Let $T$ be a local $k$-algebra of complexity at most $d$%
. Let $M$ and $M^{\prime }$ be submodules of $T^{d}$ of degree type at most $%
d$. Then the degree type of $(M:_{T}M^{\prime })=\{t\in T:tM^{\prime
}\subset M\}$ is bounded by $D$.
\end{lemma}

In particular, if $T$ is a local $k$-algebra of complexity at most $d$ and $%
J\subset T$ is an ideal of complexity at most $d$ then $\mathrm{Ann}_{T}J$
must have complexity at most $D=D(d)$, a bound that only depends on $d$.

\subsection{Proof of the asymptotic form of Koh's conjecture}

In this section we give an \textit{non standard} proof of the asymptotic
version of Koh's conjecture. We start by defining the \emph{complexity of a
ring extension }\textbf{(}$d$ will denote a positive integer).

\begin{definition}
Let $R\subset S$ be a module-finite extension of local $k$-algebras. We say
this extension has complexity $\leq d$ if:

\begin{enumerate}
\item The \emph{complexity} over $k$ of the local $k$-algebras $R$ and $S$
is at most $d$.

\item The minimal number of generators of $S$ as an $R$-module is at most $%
d. $

\item The projective dimension of $S$ as an $R$-module is less than or equal
to $d$.
\end{enumerate}
\end{definition}

We intend to prove the following: Given $d>0$, there exists a prime $p_{d}$
such that for any algebraically closed field $k$ of characteristic $p>p_{d}$%
, and any modulo-finite extension $R\subset S$ of local $k$-algebras of
complexity less than $d$, there is a retraction $\rho :S\rightarrow R$ and
consequently $R\subset S$ splits.

We start by recalling the following definition.

\begin{definition}
\label{goresnstein}A local ring $(R,\mathfrak{m})$ is called Gorenstein if $%
R $ is Cohen Macaulay (CM), and if for any system of parameters $%
\{x_{1},\ldots ,x_{d}\}$ in $R$ the socle of $\overline{R}=R/(x_{1},\ldots
,x_{d})$, defined as $\mathrm{Ann}_{\overline{R}}(\mathfrak{m})$, is a $1$%
-dimensional $R/\mathfrak{m}$-vector space ( \cite{eisen}, page 526).
\end{definition}

The following result will be of fundamental importance \cite{velez-rigo}.

\begin{proposition}
\label{retraccion} Let $(R,\mathfrak{m})\subset (T,\mathfrak{n})$ be a
module-finite extension of local rings with $T$ a free $R$-module. If $T/%
\mathfrak{m}T$ is Gorenstein, and if $J\subset T$ is any ideal, then there
exists a retraction $\rho :T/J\rightarrow R$ if and only if $\mathrm{Ann}%
_{T}(J)\nsubseteq \mathfrak{m}T.$
\end{proposition}

We also need the following \cite{velezsplitting}.

\begin{theorem}[Koh in characteristic zero]
\label{koh 0}Let $R$ be a ring containing a field of characteristic zero,
and let $R\subset S$ be a module-finite extension of rings such that the
projective dimension of $S$ as an $R$-module is finite. Then, there exists a
retraction $\rho :S\rightarrow R.$
\end{theorem}

\begin{remark}
\label{grave}Let $(R,\mathfrak{m})$ be a local ring and let $R\subset S$ be
a module-finite extension. Let us take $s_{1},\ldots ,s_{n}\in S$ generators
of $S$ as an $R$-algebra. For each $s_{i}\in S$ choose an arbitrary monic
polynomial with coefficients in $R,$ $%
f_{i}(t)=x_{i}^{d_{i}}+r_{i1}x_{i}^{d_{i}-1}+\cdots +r_{id_{i}}$, that each $%
s_{i}$ satisfies. Let $T$ denote the quotient ring $R[x_{1},\ldots
,x_{n}]/(f_{1}(x_{1}),\ldots ,f_{n}(x_{n}))$. As in Section \ref{reducciones}%
, we may represent $S$ as a quotient of $T$ by defining a surjective $R$%
-homomorphism $\phi :T\rightarrow S$, sending the class of $x_{i}$ into $%
s_{i}$. If $J$ denotes its kernel then $\mathrm{ht}(J)=0$, as showed in that
same section.

The representation of $S$ as the quotient $T/J$ makes it possible to give a
very useful criterion for the existence of a retraction. Let $(R,\mathfrak{m}%
)$ be a local ring and let $R\subset S$ be a module-finite extension. Then,
the inclusion map $R\subset T/J$ splits if and only if $\mathrm{Ann}_{T}(J)$
is not contained in $\mathfrak{m}T$. This follows immediately from
Proposition \ref{retraccion}.
\end{remark}

The following theorem states that given $i\geq 0$ and $d>0$ there exists a
formula $(\mathrm{Tor}_{i})_{d}$ such that for any $k$-algebra $R$ and $R$%
-modules $M,N,V,$ all of complexity at most $d$, then $\mathrm{Tor}%
_{i}^{R}(M,N)\cong V$ if and only if $(\mathrm{Tor}_{i})_{d}$ evaluated in
codes of $R,M,N$ and $V$ is true over $k$ (analogously for $\mathrm{Ext}$).

\begin{theorem}
\label{4.4}Given $i\geq 0,d>0$, there exist formulas $(\mathrm{Tor}_{i})_{d}$
and $(\mathrm{Ext}^{i})_{d}$ with the following properties: Let $k$ be any
field; then, if a tuple $(a,m,n,v)$ is in $|(\mathrm{Tor}_{i})_{d}|_{k}$
(respectively, in $|(\mathrm{Ext}^{i})_{d}|_{k}$), then $\mathcal{M}(v)$ is
isomorphic to \[\mathrm{Tor}_{i}^{\mathcal{A}(a)}(\mathcal{M}(m),\mathcal{M}%
(n))\] (respectively to $\mathrm{Ext}_{\mathcal{A}(a)}^{i}(\mathcal{M}(m),%
\mathcal{M}(n))$). Moreover, for each tuple $(a,m,n)$ we can find at least
one $v$ such that $(a,m,n,v)$ belongs to $|(\mathrm{Tor}_{i})_{d}|_{k}$
(respectively, to $|(\mathrm{Ext}^{i})_{d}|_{k}$).
\end{theorem}

\begin{proof}
See \cite{bounds}, Corollary 4.4, page 150.
\end{proof}

We recall the following standard result (\cite{eisen}, page 167, Theorem
6.8).

\begin{theorem}
Let $(R,\mathfrak{m})$ be a local ring, and denote by $k$ the residue field $%
R/\mathfrak{m}$. If $M$ is a finitely generated $R$-module, then \textrm{pd}$%
_{R}(M)\leq n$ if and only if $\mathrm{Tor}_{n+1}^{R}(M,k)=0$.
\end{theorem}

\begin{remark}
\label{formulita}It is clear from the previous theorems that there exists a
formula $($\textrm{pd}$_{<n})_{d}$ such that, if $M=\mathcal{M}(v)$ is an $R=%
\mathcal{R}(a)$-module with complexity less than $d$, where $(R,\mathfrak{m}%
) $ is a local $k$-algebra with complexity less than $d$, then $k\models ($%
\textrm{pd}$_{<n})_{d}(a,v),$ if and only if, \textrm{pd}$_{R}(M)\leq n$.
\end{remark}

From this preliminaries we obtain the following main result:

\begin{theorem}
For each $d>0$ there exists a first order formula $\mathrm{Koh}_{d}$ such
that if $R\subset S$ is a module-finite extension of local $k$-algebras such
that the complexity of this extension is at most $d$, then there exists a
retraction $\rho :S\rightarrow R$ if and only if $k\models \mathrm{Koh}%
_{d}(a,b),$ where $R\cong \mathcal{R}(a)$ and $S\cong \mathcal{S}(b)$.
\end{theorem}

\begin{proof}
As shown in Remark \ref{grave} we can represent $S$ as $S\cong T/J$, where
the complexity of $J$ and $T$ is less than $d$. We know there is a
retraction $\rho :S\rightarrow R$ if and only if $\mathrm{Ann}%
_{T}(J)\nsubseteq mT$.\newline
As proved in Remark \ref{formulita}, there exists a formula $($\textrm{pd}$%
_{<n})_{d}$ such that if $R=\mathcal{R}(a)$ and $S=\mathcal{R}(b)$ then $%
k\models $\textrm{pd}$_{<n}(a,b)$ if and only if \textrm{pd}$_{R}(S)<n$.%
\newline
Let $\mathrm{Koh}_{d}$ be the formula which establishes the following: if 
\textrm{pd}$_{R}(S)<d$, then $\mathrm{Ann}_{T}(J)\nsubseteq mT$. Explicitly: 
\begin{equation*}
\mathrm{Koh}_{d}(\xi ,\xi ^{\prime },\nu ,\nu ^{\prime }):\underset{i=0}{%
\overset{d-1}{\bigvee }}Pd_{R}(\xi ,\nu )=i\Longrightarrow \lnot Inc_{d}(\nu
^{\prime },\xi ^{\prime }).
\end{equation*}%
Here $\nu ^{\prime }$ and $\xi ^{\prime }$ are reserved for a code of $%
\mathrm{Ann}_{T}(J)$ and $\mathfrak{m}T$, respectively. Then, it is clear
that $k\models \mathrm{Koh}_{d}(a,a^{\prime },b,b^{\prime })$ if and only if
there exist a retraction $\rho :S\rightarrow R$.
\end{proof}

\begin{theorem}
\label{muymuybueno}Let $R\subset S$ be module-finite extensions of local $k$%
-algebras. Fix $d>0,$ an arbitrary positive integer. The set of prime
numbers $p$ for which there are counterexamples to Koh's Conjecture of
complexity less than $d$ is finite.
\end{theorem}

\begin{proof}
From Theorem \ref{koh 0} we see that $K\models \mathrm{Koh}_{d}(a,b)$ for
any field $K$ of characteristic zero. Then, by Proposition \ref{tranferencia
fuerte}, we deduce that $k\models \mathrm{Koh}_{d}(a,b),$ for every field $k$
of prime characteristic $p$ sufficiently large. More precisely: Given $d>0$,
there exists a prime number $p_{d}$ such that for any field $k$ of
characteristic $p>p_{d}$, and any modulo-finite extension $R\subset S$ of
local $k$-algebras with complexity at most $d$ there exists a retraction $%
\rho :S\rightarrow R$.
\end{proof}

\section{Acknowledgements}

Danny Arlen de Jes\'us G\'omez-Ram\'irez would like to thank most of the members of the family Ramirez Correa for the support during all these years. Specially, he thanks Brianna and Samuel (for bringing new life and happiness to our family!).


\begin{thebibliography}{10}

\bibitem{andre} Yves Andr\'e, La Conjecture du Facteur Direct, (preprint) arXiv
preprint  arXiv:1609.00345 (2016)

\bibitem{bhatt} Bhargav Bhatt, On the Direct Summand Conjecture and
Its Derived Variant, (preprint) arXiv:1608.08882 (2016)

\bibitem{hocster-Mclaughing} Dutta S. P., Hochster M., and McLaughlin J. E.,
(1985), Modules of Finite Projective Dimension with Negative Intersection
Multiplicities, Inventiones Mathematicae, 79(2) 253-291, Springer, Berlin (1985)

\bibitem{eisen} Eisenbud, D., Conmutative Algebra with a View Toward
Algebraic Geometry, Graduate Texts in Mathematics, Vol. 150,
Springer-Verlag, New York, (1994)

\bibitem{evg} Evans, E.G and Griffith, P., The Syzygy Problem,
Ann of Math, 114(2), 323-333 (1981)

\bibitem{heitmann} Heitmann R. C., A Counterexample to the Rigidity Conjecture for Rings. 
Bull. Am. Math. Soc. 29(1), 94-97 (1993)


\bibitem{heitmann2002} Heitmann R. C., The Direct Summand Conjecture in Dimension Three,
Ann of Math, 156(2), 695-712 (2002)

\bibitem{hochstercontracted} Hochster, M. Contracted Ideals from Integral
Extensions of Regular Rings, Nagoya Math.J. 51, 25-43 (1973)

\bibitem{hochsterhomological} Hochster, M., Topics in the Homological
Theory of Modules over Commutative Rings, CBMS Regional Conf. Ser. in
Math.,Vol. 24, Amer. Math. Soc., Providence (1975)

\bibitem{hochstercanonical} Hochster, M., Canonical Elements in
Local Cohomology Modules and the Direct Summand Conjecture. Journal of
Algebra, 84(2), 503-553 (1983)

\bibitem{homological conjectures old and new} Hochster, M.,
Homological Conjectures, Old and New. Illinois Journal of Mathematics,
51(1), 151-169 (2007)

\bibitem{CM} Hochster, M., \& Huneke, C., Infinite Integral
Extensions and Big Cohen--Macaulay Algebras. Annals of mathematics, 53-89 (1993)


\bibitem{thesis koh} Koh, J. H., The Direct Summand Conjecture
and Behavior of Codimension in Graded Extensions. Doctoral dissertation, Ph.
D. Thesis, University of Michigan (1983)

\bibitem{kunz} Kunz, E., Introduction to Commutative Algebra and
Algebraic Geometry. Springer, Berlin (2002)

\bibitem{modelteory} Marker, D., Model Teory: An Introduction, Graduate
texts in Mathematics, Vol 217, Springer, New York (2002)

\bibitem{mcculloughthesis} McCullough, J. (2011). On the Strong Direct
Summand Conjecture, PhD Thesis, University of Illinois, Urban-Champaigne
(2009) \url{http://math.ucr.edu/\symbol{126}jmccullo/talks/09Urbana.pdf}

\bibitem{ohidirsumcon} Ohi, T., Direct Summand Conjecture and
Dscent for Flatness. Proceedings of the American Mathematical Society,
124(7), 1967-1968 (1996)

\bibitem{Peskine} Peskine, C., \& Szpiro, L., Dimension
projective finie et cohomologie locale. Publications Math\'{e}matiques de
l'IH\'{E}S, 42(1), 47-119 (1973)

\bibitem{robertsdirsumcon} Roberts, P., Heitmann's proof of the
direct summand conjecture in dimension 3. (2002) (preprint) arXiv preprint math/0212073. 

\bibitem{bounds} Schoutens, H., Bounds in Cohomology, Israel J. Math.
116, 125-169 (2000)


\bibitem{Libro schoutens} Schoutens, H., The Use of Ultraproducts in
Commutative Algebra, Lecture Notes in Mathematics, 1999, 204p, Springer
(2010).


\bibitem{serrelocal} Serre, J. P., \& Gabriel, P., Alg\`{e}bre
locale, multiplicit\'{e}s: cours au Coll\`{e}ge de France, 1957-1958 (No.
11). Springer, Berlin (1975)

\bibitem{velezsplitting} V\'{e}lez, J. D., Splitting Results in
Module-Finite Extension Rings and Koh's Conjecture. Journal of Algebra,
172(2), 454-469. (1995)

\bibitem{velez-rigo} V\'{e}lez, J. D., \& Florez, R., Failure of
Slitting from Module-Finite Extension Rings. Beitr\"{a}ge zur Algebra und
Geometrie, 41(2), 345-357 (2000)
\end{thebibliography}
\end{document}